\documentclass[11pt]{amsart}
\usepackage{amscd}

\usepackage{amssymb}

\newcommand{\bT}{{\mathbb T}}

\newcommand{\bZ}{{\mathbb Z}}
\newcommand{\bR}{{\mathbb R}}

\newcommand{\bA}{{\mathbb A}}
\newcommand{\bC}{{\mathbb C}}
\newcommand{\bD}{{\mathbb D}}
\newcommand{\bM}{{\mathbb M}}

\newcommand{\bG}{{\mathbb G}}

\newcommand{\ep}{{\epsilon}}
\newcommand{\la}{{\langle}}
\newcommand{\ra}{{\rangle}}

\newtheorem{thm}{Theorem}[section]
\newtheorem{lemma}[thm]{Lemma}
\newtheorem{cor}[thm]{Corollary}
\newtheorem{prop}[thm]{Proposition}
\newtheorem{as}[thm]{Assumption}

\numberwithin{equation}{section}

\begin{document}
                        
\title[motivic cohomology]{motivic cohomology of twisted 
flag varieties}
 
\author{N.Yagita}

\address{ faculty of Education, 
Ibaraki University,
Mito, Ibaraki, Japan}
 
\email{ nobuaki.yagita.math@vc.ibaraki.ac.jp, }


\begin{abstract}
In this paper, we study the $mod(p)$ motivic cohomology of
twisted complete flag varieties over some restricted fields $k$.  Here
we take the field $k=\bR$ when $p=2$. When $p$ is odd,
we take $k$ such that the Milnor $K$-theory $K_i^M(k)/p=0$
for $i>3$.
\end{abstract}
\maketitle

\section{Introduction}

Let $G$ be a simply connected compact Lie group
(and $T$ its maximal torus).
For a field $k\subset \bC$, 
let $G_k$ be the corresponding split reductive group
over $k$ and $T_k$ its split torus.  Let us write $\bG_k$ a nontrivial
$G_k$-torsor.  Then  $\bG_k/B_k$ is a (twisted form of) the complete
flag variety for a Borel subgoup $B_k\supset T_k$.  Let us simply write by $X$ the flag variety $\bG_k/B_k$.
  We will study  the motivic cohomology
$H^{*,*'}(X;\bZ/p)$ for twisted $X$.

For non-twisted cases, note that it is well known
\[  H^*(G_k/B_k;\bZ/p)\cong  H^{2*}(G/T;\bZ/p)\otimes \bZ/p[\tau]\]
where $0\not =\tau\in H^{0, 1}(Spec(\bar k);\bZ/p ).$
The cohomology $H^*(G/T;\bZ/p)$ (for the topological flag
manifold) are completely known 
by Toda, Watanabe, Nakagawa  [Tod], [Tod-Wa], [Na].

In this paper, we give examples (which are the first for odd prime $p$) such that 
the ring structures of $H^{*,*'}(X;\bZ/p)$ are computed
for twisted complete flag varieties  $X$.

 Petrov, Semenov and Zainoulline [Pe-Se-Za]
show  that 
the motive 
$M(X)$ of $X$ is 
decomposed as a sum of (degree changing) 
generalized Rost motive $R(\bG_k)$, namely,
\[ M(X)\cong R(\bG_k)\otimes (\oplus_{s} \bT^{\otimes s})\]
where $\bT^{\otimes s}$ is the (degree changing) Tate motive. 

Hence we can compute $H^{*,*'}(X;\bZ/p)$ additively if
$H^{*,*'}(R(\bG_k);\bZ/p)$ can be computed, while it is a difficult problem
in general cases.  We say that a Lie group $G$ is of type $(I)$  if it is simply connected and $H^*(G;\bZ/p)$ contains 
only one even dimensional truncated polynomial 
generator.
In this case, 
$R(\bG_k)$ is the original Rost motive $R_2$ ([Ro1],[Vo1,4],[Su-Jo]) defined from
a pure  symbol in the Milnor $K$-theory $K_3^M(k)/p$),  and
their Chow groups are known ([Me-Su], [Ya1,4]).  Moreover
 we can compute also ring structure of  $CH^*(X)/p$ (see [Ya7]).

The motivic cohomology $H^{*,*'}(R_n;\bZ/p)$, of course,
depends on the field $k$.   It is known for $k=\bR$ and $p=2$.
Hence we can compute $H^{*,*'}(X;\bZ/2)$ in  this case.
We also compute its ring structure (Theorem 6.7) by using the result
from [Ya7].  
As an application, we study Barlmer's Witt group
$W^*(X)$ for $k=\bR$.

It seems no literature to compute $H^{*,*'}(R_n;\bZ/p)$ for an odd prime $p$.  Hence we write down it
(Theorem 3.3)
when $K_{n+2}^M(k)/p=0$. (Examples of such fields are
high dimensional local fields defined by K. Kato [Ka].)
 Using these, we give the ring structure of $H^{*,*'}(X;\bZ/p)$  (Theorem 3.12) when 
$K_4^M(k)/p=0$ (and of course $K_3^M(k)/p\not =0$).
We also compute the $mod(p)$ algebraic cobordism
$MGL^{*,*'}(X;\bZ/p)$ for these fields $k$.

  
The plan of this paper is the following.
In $\S 2$, we recall the motives of flag varieties.
In $\S 3$, we compute $H^{*,*'}(R_n;\bZ/p)$ when
$K_{n+2}^M(k)/p=0$, and compute $H^{*,*'}(X;\bZ/p)$
for $G$ is type $(I)$.  In $\S 4$, we study algebraic cobordism
version.  In $\S 5$ we study the $BP\la n-1\ra^*$-theory version,
which gives a short another proof for the result in $\S 4$.
In $\S 6$, we compute $H^{*,*'}(X)/2$ when $k=\bR$, and
in $\S 7$,  we study the Balmer Witt group $W^*(X)$.

\section{Lie groups and algebraic groups}

Let $p$ be a prime number and $k$ field with $k\subset \bC$.
Let $X$ be an algebraic variety over $k$ such that
$\bar X=X\otimes _k\bar k$ is cellular for the algebraic
closure $\bar k$ of $k$.
Let $H^{*,*'}(X;\bZ/p)$ (resp. $CH^*(X)=CH^*(X)_{(p)}$) be the mod($p$) motivic 
cohomology (resp. Chow ring localized at $p$) over $k$.
Note that $H^{*,*'}(\bar X;\bZ_{(p)})$ is torsion free, since $\bar X$ is cellular.

Let $a$ be  a pure symbol  in the mod $p$  Milnor $K$-theory
$ K_{n+1}^M(k)/p$.  By Rost,  we can
construct the norm variety $V_a$ of $dim(V_a)=p^n-1$ such that $\bar V_a\cong v_n$ is cellular (where $v_n$ is a generator of $BP^*$), and
$ a=0\in k_*^M(k(V_a))/p$ for the function filed $k(V_a)$ of $V_a$.

Rost and Voevodsky  showed that
there is $y\in CH^{b_n}(\bar V_a)$ for
$b_n=(p^n-1)/(p-1)$  such that $y^{p-1}$ is the fundamental class of $\bar V_a$.  We know
([Ro1], [Vo1,4], [Me-Su], [Ya4])
\begin{thm}
For a nonzero pure symbol $a\in K_{n+1}^M(k)/p$,
 there is an irreducible (Rost) motive  $R_a$ ( write by $R_n$
simply)
 in the motive $M(V_a)$ of the norm variety $V_a$
such that
$\  CH^*(\bar R_a)\cong \bZ_{(p)}[y]/(y^p), \ $ and
\[ CH^*(R_a)
 \cong \bZ_{(p)}\{1\}\oplus_{i=1}^{p-1}M_n(y^i).\]
Here $ M_n(y^i)=\bZ_{(p)}\{c_0(y^i)\}\oplus \bZ/p\{c_1(y^i),....,c_{n-1}(y^i)\}$ with $|c_j(y^i)|=|y^i|-2(p^i-1)$,
where $\bZ/p\{a,b,...\}$ is a $\bZ$-free module generate by 
$a,b,...$
\end{thm}

Let $G$ be a simply connected compact Lie group
(and $T$ its maximal torus)
and $G_k$ be the corresponding split reductive group
over $k$ and $T_k$ its split torus.  Let $\bG_k$ be a nontrivial
$G_k$-torsor, and $\bG_k/B_k$ the (twisted form of) complete
flag variety. Let us simply write by $X$ the twisted flag variety $\bG_k/B_k$,
  We are interesting in the motivic cohomology
$H^{*,*'}(X;\bZ/p)$. (For a Borel subgroup $B\supset T$,
note $G_k/B_k$ is cellular and $H^{*,*'}(G_k/B_k;\bZ/p)\cong H^{2*}(G/T;\bZ/p)$.)

We only consider here the cases
\[ gr H^*(G;\bZ/p)\cong \bZ/p[y]/(y^p)\otimes 
\Lambda(x_1,...,x_{\ell})\]
where the degree are $|y|=even$, and $|x_i|=odd$.
The following simple Lie groups satisfy 
the above isomorphism ;
\[ (G,p)=\begin{cases}   G_2,\ F_4,\ E_6, \ Spin(7),\ Spin(8),\ Spin(9)\quad for \ p=2\\ F_4,\ E_6,\ E_7\quad for \ p=3\\
E_8\quad p=5.
\end{cases}\]
We call that these (simply connected) groups $G$ are of type $(I)$.
It is known  $\ell\ge 2(p-1)$.

There is a fibering
$G\to G/T\to BT$ and the induced spectral sequence
(see for details, [Tod]. [To-Wa], [Mi-Ni], [Na])
\[ E_2^{*,*'}\cong H^*(G;\bZ/p)\otimes H^*(BT)\]
\[\cong \bZ/p[y]/(y^p)\otimes\Lambda(x_1,...,x_{\ell})\otimes
S(t)\Longrightarrow
H^*(G/T;\bZ/p)\]
where $S(t)=H^*(BT)\cong \bZ_{(p)}[t_1,...,t_{\ell}]$ with $|t_i|=2$.  Then it is known ([Tod]) that
there is a regular sequence $(b_1,...,b_{\ell})$ in $S(t)$
such that $d_{|x_i|+1}(x_i)=b_i$ for the differential in the spectral sequence.
Hence we can write
\[ grH^*(G/T;\bZ/p)\cong \bZ/p[y]/(y^p)\otimes
S(t)/(b_1,..., b_{\ell}).
\]
Let us write
$S(t)/(p,b_1,...,b_{\ell})$ simply by $S(t)/(b)$.
\begin{thm} ([Tod])  We have the ring isomorphism
\[ H^*(G/T;\bZ/p)\cong (\bZ/p[y]\otimes S(t)/(b,f_p),\quad 
f_p=y^p\ mod(S(t)^+).\]
\end{thm}

Petrov, Semenov and Zainoulline [Pe-Se-Za] prove that
when $G$ is a  group of type $(I)$, the motive 
$M(X)$ (which is localized at $p$) is 
decomposed as a sum of (degree changing) Rost motive $R_2$, namely,
\[ M(X)\cong R_2\otimes (\oplus_{s} \bT^{\otimes s}).\]
where $\bT^{\otimes s}$ is the (degree changing) Tate motive.  Hence we easily see that
$CH^*(\oplus_s\bT^{\otimes s})/p\cong S(t)/(b)$, and 
\[ H^{*,*'}(\oplus_s\bT^{\otimes s};\bZ/p)\cong H^{*,*'}(pt.;\bZ/p) \otimes S(t)/(b).\]
Hence additively
$H^{*,*'}(X;\bZ/p)\cong H^{*,*'}(R_2;\bZ/2)\otimes S(t)/(b).$
In particular, we have additively
\[ CH^*(X)/p\cong \bZ/p\{1,c_{0}(y),c_{1}(y),...,c_{0}(y^{p-1}),
c_{1}(y^{p-1})\}\otimes S(t)/(b).\]

On the other hand, the following additive isomorphism
is immediate
\[  S(t)/(  b_i b_j, b_{k})
\cong \bZ/p\{1, b_1,..., b_{2p-2}\}
\otimes S(t)/(b)\]
where $i,j,k$ range $1\le i,j\le 2p-2< k\le \ell$. 
We can take $c_1(y^i)=b_{2i-1}$ and $c_0(y^i)=b_{2i}$,
and prove the following theorem.
\begin{thm} ([Ya7])
Let $G$ be a simple Lie group of type $(I)$.
Let $X=\bG_k/T_k$ for a nontrivial torsor $\bG_k$.
Then we have the ring isomorphism
\[ CH^*(X)/p\cong  S(t)/(p,  b_ib_j, b_{k}).\]
where $i,j,k$ range  $ 1\le i,j\le 2(p-1)<k\le \ell.$
\end{thm}

\section{motivic cohomology of the Rost motive}

Recall $V_a$ is a norm variety for a nonzero symbol $a\in K_{n+1}^*(k)$.
Let $\chi_a=\check C(V_a)$ be the Cech simplex sheaf for $V_a$ and
$\tilde \chi_a$ be the object defined by the cofibering
$\tilde \chi \to \chi\to pt$ in the stable $\bA^1$-homotopy category. 

By the solution of Bloch-Kato conjecture by Voevodsky, we see  the  exact sequence
\[(3.1)\quad 0 \to H^{*+1,*}(\chi _a;\bZ/p)
\stackrel{\times \tau}{\to }K_{*+1}^M(k)/p 
\to K_{*+1}^M(k(V_a))/p\]
identifying $H^{*+1,*+1}(\chi_a;\bZ/p)\cong K_{*+1}^M(k)$.
Since $a=0\in K_{n+1}^M(k(V_a))/p$, there is unique element 
$a'\in H^{n+1,n}(\chi_a;\bZ/p)$ such that $\tau a'=a$.

Let $M_a$ be the object in the derived category $\bD\bM$
of mixed motives [Vo2,4] defined by 
the following distinguished triangle
\[(3.2)\quad M(\chi_a(b_n))[2b_n]\to M_a  \to M(\chi_a)
\stackrel{\delta_a=Q_0...Q_{n-1}(a')}{\to} M(\chi_a)(b_n)[2b_n+1] \]
where  $b_n=(p^n-1)/(p-1)=p^{n-1}+...+p+1$
so that $deg(\delta_a)=(2b_n+1,b_n)$.
For $i<p$, define the symmetric powers
\[M_a^{i}=S^{i}(M_a)=q_i(M_a^{\otimes i})\subset M_a^{\otimes i}\]
where $q_i=(1/i!)\sum_{\sigma \in S_{i}}\sigma$
is a projector in $\bD\bM$,  and $\sigma$ 
is the symmetric group of $i$ letters.
One of the important results in [Vo2,4] Voevodsky proved is that 
$M_a^{p-1}$ is a direct summand of a motive of $V_a$
(for details see [Vo2,4]). 
This $M_a^{p-1}$ is written by $R_a$ (and is called the Rost
motive) in preceding sections.

Hence there are distinguished triangles ( (5.5),(5.6) in [Vo4])
\[(3.3)\quad M_a^{i-1}(b_n)[2b_n]\to M_a^i \to  M(\chi_a) 
\stackrel{s_i}{\to} M_a^{i-1}(b_n)[2b_n+1]\]
\[(3.4)\quad M(\chi_a)(b_ni)[2b_ni] \to M_a^i \to M_a^{i-1}
\stackrel{r_i}{\to}
M(\chi_a)(b_ni)[2b_ni+1].\]


For each prime $p$, we have the Milnor operation
\[ Q_i: H^{*,*'}(X;\bZ/p)\to H^{*+2p^i-1,*'+p^i}(X;\bZ/p).\]
 Let us write by $Q(n)$ the exterior algebra
$\Lambda(Q_0,...,Q_0)$. (We see $Q_iQ_j=-Q_jQ_i$ for all $p$ and 
$Q_i^2=0$ also for $p=2$.)

\begin{thm} (Theorem 8.5 in [Ya4])
 Let $0\not =a=(a_0,...,a_n)\in K_{n+1}^M(k)/p$.  Then 
there is a $K_*^M(k)\otimes Q(n)$-modules isomorphism
\[H^{*,*'}(\tilde{\chi}_a;\bZ/p)\cong K_*^M(k)/(Ker(a))\otimes 
Q(n)\otimes \bZ/p[\xi_{a}]
\{a'\}\]
where $\xi_{a}=Q_nQ_{n-1}....Q_0(a')$ and  $deg(a')=(n+1,n)$.
\end{thm}
 
By using the above theorem we have Theorem 2.1   
   in $\S 2$.
     \begin{cor} Let $c_i=Q_0...\hat Q_i...Q_{n-1}(a')$ (hence $|c_i|=2(b_n-p^i+1)$).
  Then
     there is the additive isomorphism
   \[ CH^*(M_a^{p-1})/p\cong \bZ/p\{1\}\oplus 
   \bZ/p[t]/(t^{p-1})\{c_0,...,c_{n-1}\}\]
where $deg(t)=(2b_n,b_n)$.
   \end{cor}
{\bf Remark.} Note that $t^i$ is a vertical element, i.e.,
$t^i$ itself  does not exist in $CH^*(M_a^{p-1})/p$ but $c_jt^i$ exists.

Let $K_1^M(k)/p$ be finite, e.g.,
generated by $a_1,...,a_m$.  When $p$ is odd,
 $K_*^M/p$ is isomorphic
to a quotient of the exterior algebra $\Lambda(a_1,...,a_m)$.
So $K_i^M(k)/p=0$ for all $i>m$.
   Hereafter we consider some ease cases such that $K_*^M(k)/p=0$ for $*>n+1$
and hence $K_*^M(k)/(Ker(a))=\bZ/p$.

 Let us write 
$H^{*,*'}=H^{*,*'}(Spec(k);\bZ/p)\cong K_*^M(k)/p[\tau]$
and 
\[\tilde H^{*,*'}=H^{*,*'}-\bZ/p\{1\}= \oplus_{(*,*')\not =(0,0)}H^{*,*'}(Spec(k),\bZ/p).\]
We assume that a primitive $p^2$-th root of unity exists
in $k$, that is each  primitive $p$-th root $\zeta_p$ is zero
in $K_1^M(k)/p$.  Hence we see
\[ Q_0(\tau)=\zeta _p=0\quad in\  H^{*,*'}(pt.;\bZ/p).\] 

We have the additive structure of $H^{*,*'}(M_a^{p-1};\bZ/p)
\cong H^{*,*'}(R_n;\bZ/p)$ by the following theorem.
 
          \begin{thm}
Let $k$ contain a primitive $p^2$-th root of the unity.
       Let $a\in K_{n+1}(k)/p$ be a nonzero pure 
symbol, and $K_*^M(k)/p=0$ for $*>n+1$. 
       Then there is an additive isomorphism
       \[H^{*,*'}(R_n;\bZ/p)\cong
          H^{*,*'}\oplus 
             (\tilde  Q(n-1)\{a'\}\oplus \tilde H^{*,*'}\{t\})
\{1,t,...,t^{p-2}\}\]
where  $\tilde Q(n-1)=Q(n-1)-\bZ/p\{Q_0...Q_{n-1}\}.$
           \end{thm}
Recall that $H_{et}^*(X;\bZ/p)\cong H^{*,*}(X;\bZ/p)$.
Since $\tau t\in \tilde H^{*,*'}\{t\}$ and $\tau Q_i(a')=Q_i(a)=0$, we can see 
\begin{cor}   We have $ H_{et}^*(R_n;\bZ/p)\cong K_*^M(k)[t]/(p,t^p).$
\end{cor}
This fact itself is easily seen from $H^{*}_{et}(\chi_a;\bZ/p)
\cong H^*_{et}(pt.;\bZ/p)$.
Let us write by $\tilde H_{et}^*(R_n;\bZ/p)$ the same ring
$H_{et}^*(R_n;\bZ/p)\cong K_*^M(k)[t]/(t^p)$  with $deg(t)=(2b_n,b_n)$
(but not $(2b_n,2b_n))$.

\begin{cor}  We have the exact sequence
\[ 0\to \tilde Q(n-1)\{a'\}\otimes \bZ/p[t]/(t^{p-1})
\to H^{*,*'}(R_n;\bZ/p)\qquad \qquad \]
\[\qquad \qquad \to \tilde H^*_{et}(R_n;\bZ/p)\otimes \bZ/p[\tau]\cong H^{*,*'}[t]/(t^p)
\to \bZ/p[t]^{+}/(t^p)\to 0.\]
\end{cor}

For ease of notations, for $x\in H^{*,*'}(Y;\bZ/p)$,
let us write by  $f.deg(x)=*,$ and  $s.deg(x)=*'$,
and define 
\[ w(x)=2s.deg(x)-f.deg(x),\qquad d(x)=f.deg(x)-s.deg(x).\]
Then it is well known ([Vo1,2]) that if $Y$ is smooth and 
$0\not =x\in H^{*,*'}(Y;\bZ/p)$, then
$0\le w(x)$ and $0\le d(x)\le dim(Y).$

{\bf Remark.}
   Let us consider the following triangular domain 
in $\bZ\times \bZ$ generated by bidegree $(*,*')$
   \[D_i'=\{deg(x)|0\le w(x),\ 2b_n(i-1)\le  *<2b_ni,\ 
    b_n(i-1)<   d(x)\}\]
   and $D'=\cup_{j=1}^{p-1}D_j'$.
Then we know in  Lemma 8.7 in [Ya4]  the following 
fact.  Let us write $K=K_*^M(k)$,  $ _aK=Ker(a))$ in $K$ and $K_a=K/K_a$.  Then for bidegree $(*,*')\in D'$ 
   defined above,
   we have the $K$-module isomorphism,  
    \[H^{*,*'}(R_n;\bZ/p) \cong K\{1\}\oplus 
 (K_a\otimes Q(n-1)\{a'\}\oplus\ _aK\{t\})\otimes 
    K[t]/(t^{p-1}).\]
So the above isomorphism shows the preceding theorem
when $(*,*')\in D'$, since $K\cong H^{*,*'}$ in $D'$.

To prove the above theorem, we need some lemmas and arguments.  Let us assume that $p$ is  an odd prime ($p=2$ case is proved similarly).  Let us write
$\chi_a$ (resp. $M_a^i, \delta_a,\xi_a$) by simply
$\chi$ (resp. $M^i,\delta,\xi$).
Define 
\[D_i(Y)=\{H^{*,*'}(Y;\bZ/p)|2b_ni+n< *
\le 2b_n(i+1)+n\} \]
so that $H^{*,*'}(Y;\bZ/p)=\sum _{i=-1}^{\infty}D_i(Y)$
and $D_i\cap D_j=0$ for $i\not =j$.
Note $|\xi|=2b_{n+1}=2pb_n$. 
To see arguments
easily, here  we use notations that
$\begin{CD} @>{D_i}>a_i> or  @>{a_i}>{D_i}> \end{CD}$
simply means $a_i\in D_i$, e.g.,
\[\begin{CD}
   H^{*,*'}(\chi;\bZ/p): @>{D_{-1}}>{1}>   @>{D_{0}}>{a',\delta}> @>{D_{1},...}>{0,...}> @>{D_{p-2}}>{0}> @>{D_{p-1}}>{Q_n(a'),\xi}> @>{D_{p}}>{a'\xi,\delta\xi}>
\end{CD}  \]
 From Theorem 3.1 and
$K_*^M(k)/Ker(a))\cong \bZ/p$, we have
$ D_i(\chi)\cong D_{i+pm}(\chi)$  for $\ i\ge 1,$
and $D_i(\chi)\not=0$ only 
if $i=-1,0\ mod(p)$.

On the other hand, when $Y=M^{p-1}$, we see
$D_i=0$ for $i\ge p-1$ 
\[\begin{CD}
   H^{*,*'}(M^{p-1};\bZ/p): 
@>{D_{-1}}>{H^{*,*'}}>  
 @>{D_{0}}>{a',\tilde H^{*,*'}t}>
 @>{D_{1},...}>{a't,\tilde H^{*,*'}t^2,...}> 
@>{D_{p-2}}>{a't^{p-2}, \tilde H^{*,*'}t^{p-1}}> 
@>{D_{p-1}, D_{p},...}>{0,0,...}>
\end{CD}  \]

We first compute $H^*(M;\bZ/p)$.
 Consider the exact sequence induced from (3.3)
   \[... \stackrel{\delta}{\gets} H^{*,*'}
   (\chi(b_n)[2b_n];\bZ/p)
   \stackrel{j_1}{\gets}H^{*,*'}(M;\bZ/p)
   \stackrel{j_2}{\gets}  H^{*,*'}(\chi;\bZ/p)
\stackrel{\delta}{\gets}...\]
Hence we have the isomorphism
\[ H^{*,*'}(M;\bZ/p)\cong Coker(\delta)\oplus
Ker(\delta)[b_n](2b_n).\]

The map $\delta: \chi\to \chi(b_n)[2b_n+1]$ in $\bD\bM$
is given by
the cup product $x\mapsto x\cup \delta$ for an element
$\delta\in H^{2b_n*1,b_n}(\chi;\bZ_{(p)})$.  In fact, we note 
\[ Hom_{\bD\bM}(M(\chi),\ M(\chi)(*')[*])
\cong Hom_{\bD\bM}(M(\chi)\otimes M(\chi),\ \bZ_{(p)}(*')[*])\]
\[ \cong Hom_{\bD\bM}(M(\chi),\ \bZ_{(p)}[*')[*])  \cong
H^{*,*'}(\chi;\bZ_{(p)}).\]

Note that from $Q_i^2=0$, the map multiplying $ \delta=Q_{n-1}...Q_0(a')$ is a $Q(n-1)$-module map,
while is not a $Q_n$-map.

\begin{lemma}In $H^{*,*'}(\chi;\bZ/p)$, we have
$\delta\cdot Q_n(a')=\xi\cdot a'$.
\end{lemma}
\begin{proof}  We have
\[ Q_n(\delta)\cdot Q_n(a')=\xi\cdot Q_n(a')=
Q_n(\xi\cdot a').\]
So $\delta\cdot Q_n(a')=\xi\cdot a'+Q_n(b)$ for some $b$ since $H^{*,*'}(\tilde\chi;\bZ/p)$ is $Q(n)$-free.
But  $ b=\sum \lambda_JQ_J(a')$ for some $\lambda_J\in\bZ/p$
and $Q_J\in Q(n-1)$. Since $|b|=|\delta a'|$,
we see $\lambda=0\in\bZ/p$.
\end{proof}
The map $\times \delta : H^{*,*'}(\chi;\bZ/p)\to
H^{*,*'}(\chi;\bZ/p)$ is written as the following 
\[ \begin{CD}
   @. @>{D_{-1}}>{\ \ 1}>  
 @>{D_{0}}>{a',\delta}>
 @>{D_{1},...,D_{p-3}}>{0,...,0}>
 @>{D_{p-2}}>{0}> 
@>{D_{p-1}}>{Q_n(a'),\ \xi}> 
@>{D_{p}}>{a'\xi,\ \delta\xi}> \\
 @. @V{\delta}VV  @V{\delta}VV   @V{\delta}VV  @V{\delta}VV @V{\delta}VV  @V{\delta}VV \\
  @>{1}>{D_{-1}}> 
  @>{a',\delta}>{D_0}>
 @>{0,0}>{D_{1}}>
@>{0,...,0}>{D_2,...,D_{p-2}}> 
@>{Q_n(a'),\ \xi}>{D_{p-1}}>
 @>{a'\xi,\ \delta\xi}>{D_{p}}>
@>{0,\ 0}>D_{p+1}>,
\end{CD}  \]
which is not a commutative diagram, and it means only
$\delta\cdot 1=\delta$,  $\delta\cdot a'=0$, $\delta \cdot \delta=0,$ .., $\delta\cdot Q_n(a')=a'\xi$,...

Recall that $Q(n-1)\{a'\}/\delta\cong \tilde Q(n-1)\{a'\}$.  We also recall
 $\delta\cdot Q(n-1)Q_n\{a'\}\cong Q(n-1)\{a'\xi\}$
from the above lemma. 
 Using these, we can compute
\[ Coker(\delta)\cong H^{*,*'}\oplus 
\tilde Q(n-1)(a')
 \oplus Q(n-1)Q_n(a')\otimes \bZ/p[\xi],\]
\[ Ker (\delta)[b_n](2b_n)\cong Ker(\delta)\{t\}\cong \tilde H^{*,*'}\{t\}
\oplus (Q(n-1)(a't)\otimes\bZ/p[\xi]\]
where $t$ is a vertical element with $deg(t)=(2b_n,b_n)$.
  We can write $H^{*,*'}(M;\bZ/p)$ as 
\[ \begin{CD}
  @>{H^{*,*'}}>{D_{-1}}> 
  @>{a', \tilde H^{*,*'}t}>{D_0}>
 @>{a't,\delta t}>{D_{1}}>
@>{0,...,0}>{D_2,...,D_{p-2}}> 
@>{Q_n(a'),\xi}>{D_{p-1}}>
 @>{0,0}>{D_{p}}>
@>{a'\xi t, \delta \xi t}>D_{p+1}>
\end{CD}  \]
Here note $Q_{n-1}...Q_0(Q_n(a'))=\xi$.
\begin{lemma}
We have the isomorphism 
\[ H^{*,*'}(M;\bZ/p)\cong 
H^{*,*'}\oplus \tilde Q(n-1)(a')\oplus \tilde H^{*,*'}\{t\}\]
\[ \oplus 
 Q(n-1)\{a't,Q_n(a')\}\otimes \bZ/p[\xi].\]
 \end{lemma} 

Next, we compute $H^{*,*}(M^i;\bZ/p)$ for $1<i<p-1$.
We consider  the long exact sequence
induced from (3.4)
\[...\stackrel{r_i^*}{\gets} H^{*-2b_ni,*'-b_ni}(\chi;\bZ/p)\gets
    H^{*,*'}(M^i;\bZ/p)\gets H^{*,*'}(M^{i-1};\bZ/p)
\stackrel{r_i^*}{\gets}...\]
Recall that $D_j(\chi)=0$ for $j\not =-1,0\ mod(p)$.
Hence 
\[ H^{*,*'}(M^i;\bZ/p)\cong H^{*,*'}(M^{i-1};\bZ/p)
\quad for \ i\not =-1,0\ mod(p).\]
The element  $r_i^*$ is represented by an element in $H^{2b_ni+1,b_ni}(M^{i-1};\bZ/p)$ from
\[ Hom_{\bD\bM}(M^{i-1},M(\chi)(*')[*])\cong Hom_{\bD\bM}(M^{i-1}\otimes M(\chi), \bZ/p(*')[*])\]
\[ \cong Hom_{\bD\bM}(M^{i-1},\bZ/p(*')[*])\cong  H^{*,*'}(M^{i-1};\bZ/p).\]

Hence $H^{*,*}(M^i;\bZ/p)\cong Coker(r_i^*)\oplus
Ker(r_i^*)\{t^i\}$. We first study in the case $i=2<p-1$.
Note that the map $\times \delta t=\times Q_{n-1}...Q_0(a't)$ is a $Q(n-1)$-module map.
\begin{lemma} we have $r_2^*(1)=\delta t$
and $r_2^*(Q_n(a'))=a't\xi$. (I.e. $\delta tQ_n(a')=a't\xi$.)
\end{lemma}
\begin{proof}
By the exact sequence, we know
\[H^{*,*'}(M^i;\bZ/p)\cong H^{*,*'}(M^{i-1};\bZ/p)
\quad for\ *<2b_ni.\]
Therefore if $0\not =\delta t\in H^{*,*'}(M^2;\bZ/p)$, then so is in $H^{*,*'}(M^{p-1};\bZ/p)$ which is a sub-motive of $V_a$.
This is a contradiction, since
$ w(\delta)=(2s.deg-f.deg)(\delta)=-1<0.$

The case $r_2^*(Q_n(a')$ is proved similarly, by using
$w(a'\xi)<0$.
\end{proof}

{\bf Remark.}  This lemma is immediate consequence 
from Lemma 3.5, if $t$ exists, however, which is not correct. 

The map $r_2^*: H^{*,*'}(\chi;\bZ/p)\to H^{*,*'}(M;\bZ/p)$ is given as follows
\[\begin{CD}
 @.  @>{D_{-1}}>{\ \ 1}>   @>{D_{0}}>{a',\delta}> 
@>{D_{1},...,D_{p-4}}>{0,...,0}>@>{D_{p-3}}>{0}>
 @>{D_{p-2}}>{0}> @>{D_{p-1}}>{Q_n(a'),\xi}> \\
 @. @V{r_2^*}VV @V{r_2^*}VV @V{r_2^*}VV @V{r_2^*}VV
@V{r_2^*}VV @V{r_2^*}VV\\
  @>{H^{*,*'},a', \tilde H^{*,*'}t}>{D_{-1},D_0}>
 @>{a't,\delta t}>{D_{1}}>
@>{0,0}>{D_2}>
@>{0,...,0}>{D_3...,D_{p-2}}> 
@>{Q_n(a'),\xi}>{D_{p-1}}>
 @>{0}>{D_{p}}>
@>{a'\xi t, \delta \xi t}>{D_{p+1}}>
\end{CD}  \]
Using this lemma, we can compute
\[ Coker(r_2^*)\cong H^{*,*'}\oplus 
\tilde Q(n-1)(a') \oplus \tilde H^{*,*'}\{t\}\oplus
\tilde Q(n-1)\{a't\}\]
\[ 
 \oplus Q(n-1) Q_n(a')\otimes \bZ/p[\xi],\]
\[ Ker (r_2^*)\{t^2\}\cong \tilde H^{*,*'}\{t^2\}
\oplus Q(n-1)\{a't^2\}\otimes \bZ/p[\xi].\]
For $Y=M^2$, we see
\[ \begin{CD}
   @>{H^{*,*'}}>{D_{-1}}> 
  @>{a', \tilde H^{*,*'}t}>{D_0}>
 @>{a't,\tilde H^{*,*'}t^2}>{D_{1}}>
 @>{a't^2,\delta t^2}>{D_{2}}>
@>{0,...,0}>{D_3...,D_{p-2}}> 
@>{Q_n(a'),\xi}>{D_{p-1}}>
 @>{0, 0}>{D_{p},  D_{p+1}}>
@>{a'\xi t^2, \delta \xi t^2}>{D_{p+2}}>
\end{CD}  \]

Thus we get $H^{*,*'}(M^2;\bZ/p)$.  
By using induction for $0\le i\le p-3$, we have $H^{*,*'}(M^{p-2};\bZ/p)$
also.   We have $H^{*,*'}(M^{p-2};\bZ/p)$ 

\[\begin{CD}
 @>{H^{*,*'}}>{D_{-1}}> 
  @>{a't^i, \tilde H^{*,*'}t^{i+1}}>{D_i}>
 @>{a't^{p-2},\delta t^{p-2}}>{D_{p-2}}>
@>{Q_n(a'),\xi}>{D_{p-1}}>
 @>{0,0}>{D_{p+i}}>
@>{a'\xi t^{p-2}, \delta \xi t^{p-2}}>{D_{2p-2}}>
@>{Q_n(a')\xi,\xi^2}>{D_{2p-2}}>
\end{CD} \]
\begin{lemma}
We have the isomorphism 
\[ H^{*,*'}(M^{p-2};\bZ/p)\cong 
H^{*,*'}\oplus (\tilde Q(n-1)(a')\oplus \tilde H^{*,*'}\{t\})
\{1,t,...,t^{p-3}\}\]
\[  \oplus Q(n-1)\{a't^{p-2}, Q_n(a')\}\otimes \bZ/p[\xi].
\]
 \end{lemma}

At last we compute $H^{*,*'}(M^{p-1};\bZ/p)$
by using the exact sequence induced from (3.3) for $i=p-1$
and the following lemma. In particular 
$Ker(r_{p-1}^*)\{t^{p-1}\}\cong \tilde H^{*,*'}\{t^{p-1}\}$. 
\begin{lemma}
The map $r_{p-1}^*$ is given by
$1\mapsto \delta t^{p-2},$ 
\[ a'\mapsto Q_n(a'), \quad Q_n(a')\mapsto a'\xi t^{p-2},
\quad a'\xi\mapsto Q_n(a')\xi.\] 
\end{lemma}
\begin{proof}
All elements in the right hand side in the above maps
must  disappear in $H^{*,*'}(M^{p-1})$.
For example,
\[ d(Q_n(a'))=(f.deg-s.deg)(Q_n(a'))=(2p^n-1-p^n)+(n+1-n)\]
\[=p^n
>p^n-1=dim(V_a).\]
\end{proof}

{\bf Remark.}  There seems a relation something like 
$\delta\cdot (\delta t^{p-2})=\xi$.

The map $r_{p-1}^*:H^{*,*'}(\chi;\bZ/p)\to H^{*,*'}(M^{p-2};\bZ/p)$ is given as follows
\[\begin{CD}
  @. @.  @>{D_{-1}}>{\ \ 1}>  
 @>{D_{0}}>{a',\delta}>
 @>{D_{i+1}}>{0,0}>
@>{D_{p-1}}>{Q_n(a'),\xi}>
\\
 @. @. @V{r_{p-1}^*}VV 
@V{r_{p-1}^*}VV 
@V{r_{p-1}^*}VV 
@V{r_{p-1}^*}VV 
\\
  @>{H^{*,*'}}>{D_{-1}}> 
  @>{a't^i, \tilde H^{*,*'}t^{i+1}}>{D_i}>
 @>{a't^{p-2},\delta t^{p-2}}>{D_{p-2}}>
@>{Q_n(a'),\xi}>{D_{p-1}}>
 @>{0,0}>{D_{p+i}}>
@>{a'\xi t^{p-2}, \delta \xi t^{p-2}}>{D_{2p-2}}>
\end{CD}  \]
where $1\le i\le p-3$.
The map $r_{p-1}^*$ is injective only except
for $H^{*,*'}\{1\}$ (namely for $*>n+1$). 
We see that $Ker(r_{p-1}^*)\cong \tilde H^{*,*'}t^{p-1}$.
Then we have $H^{*,*'}(M^{p-1})\cong Coker(r_{p-1}^*)\oplus
H^{*,*'}t^{p-1}$.  So we have proved Theorem 3.3.

\begin{lemma}
Let $ b=Q_{n-1}...Q_1(a't^{i-1})$
in $H^{*,*'}(M^{p-1};\bZ/p)$.  Then $b$ can be lift to
the element 
$H^{*,*'}(M^{p-1};\bZ/p^2)$  with 
$b=pt^i$.
\end{lemma}
\begin{proof}
 Recall the exact sequence 
\[ ... \gets H^{*,*'}(M^{i-1};\bZ/p)
\stackrel{r_i^*}{\gets} 
H^{*,*'}(\chi;\bZ/p) \stackrel{\delta_i}{\gets} 
H^{*,*'}(M^i;\bZ/p)\gets ... \]
The element $ht^i\in H^{*,*'}(M^i;\bZ/p)$ is defined by
$\delta_i^*(ht^i)=h$ when $r_i^*(ht^{i-1})=\delta t^{i-1}h
=0$.

 Since $\delta t^{i-1}$ is the $Q_0$-image, we know $p\delta t^{i-1}=0$ in 
$H^{*,*'}(M^i;\bZ/p^2)$.  Hence there is an element
$pt^i\in H^*(M^i;\bZ/p^2)$, which is $\bZ/p^2$-module generator (since $t^i$ itself does not exist).
So $pt^i$ represents nonzero element in $H^{*,*'}(M^i;\bZ/p)$.

On the other hand, from Theorem  3.3, we see  $H^{2b_ni,b_ni}(M^i;\bZ/p)\cong \bZ/p\{b\}$.  Hence we can take $pt^i=b$. 
\end{proof}

 Let $G$ be a Lie group of type  $(I)$ and $X=\bG_k/B_k$.  
 Using the case $n=2$, we see that 
$H^{*.*'}(X;\bZ/p)$ is additively isomorphic to 
\[(H^{*,*'}\oplus  (\bZ/p\{a',Q_0(a'),Q_1(a')\}\oplus \tilde H^{*,*'}\{y\}
 )\otimes \bZ/p[y]/(y^{p-2}))\otimes  S(t)/(b).\]  
Here we have $Q_0(a'y^{i-1})=b_{2i}$, $Q_1(a'y^{i-1})=b_{2i-1}$.
By the preceding lemma, we can take the $i$-th product $y^i$ in $H^{*,*'}(\bar X;\bZ/p)$ as the additive
generator $t^i$ in $H^{*,*'}(\bar M^{p-1};\bZ/p)
\cong \bZ/p[y]/(y^p)$.

\begin{thm}
Let $k$ be a field in Theorem 3.3.
Let  $G$ be a group of type of $(I)$, 
and $X=\bG_k/B_k$ be a twisted flag variety.
Then $H^{*,*'}(X;\bZ/p)$ is isomorphic to the
$H^{*,*'}$-subalgebra generated by 
\[1,\ \  a'y^i,\ \ \tilde H^{*,*'}y^{i+1} \ (0\le i\le p-2),\ \
and \ \ S(t) \] 
in $H^{*,*'}[1,a',y]\otimes S(t)/(R_1',R_2',R_3')$  where
\[ R_1'=R_1=(b_ib_j,b_k|1\le i,j\le 2p-2<k\le \ell),\]
\[ R_2'=(K_{+}^Ma', \tilde H^{*,*'}b_i, a'b_i| 1\le i\le 2p-2),\quad \]
\[ R_3'=(\tau a'-a,  (a')^2, \tilde H^{*,*'}y^p).\]
\end{thm}  
\begin{proof}

The relation $R_1$ holds, since so does in $CH^*(X)$.
Since $K_+^M(k)a=0$, we have $K_+^M(k)a'=0$ by using (3.3).  Then we have 
\[ K_+^M(k)b_i=K_+^M(k)Q_i(a')=Q_i(K_+^M(k)a')=0.\]
We also have
$ \tau b_i=\tau Q_i(a')=Q_i(\tau a')=Q_i(a)=0.$
Since $f.deg(a')=3$, we have $(a')^2=0$ for $p$ odd,
and so  $b_ia'=Q_i(a')a'=1/2Q_i((a')^2)=0$.  
When $p=2$, we have $a^2=a\rho=0$, and
$b_ia'=0$ from $|b_ia'|>|y|$.
Thus we have $R_2$.

For $R_3$, we only see the last element.  We have
in $CH^*(X)/p$
\[ y^p=\sum b_it(i)\quad    t(i)\in S(t).\]
Therefore from $R_2$, we see $\tilde H^{*,*'}y^p=0$.
\end{proof}    
   \begin{cor}  We have an isomorphism
\[ grH_{et}^{*}(X;\bZ/p)\cong K_*^M(k)/p\otimes \bZ[y]/(y^{p})\otimes S(t)/(b).\]
\end{cor}

  {\bf Remark.}
   As examples of the fields  satisfying the assumption of 
Theorem 3.3,
   we can take the high dimensional local fields defined by Kato and 
   Parsin. 
      Let $k$ be a complete discrete valuation field  with residue field $F$.
   Then it is well known that
   \[K_r^M(k)/p\cong K_r^M(F)/p\oplus K_{r-1}^M(F)/p.\] 
      Let $k_0$ be a finite field, and let $k_1,...,k_n$ be the sequence 
   of complete discrete valuation fields such that the residue field of 
   $k_i$ is $k_{i-1}$ for each $1\le i\le n$.  Then the field $k=k_n$ obtained 
   in this way is called an $n$-dimensional local field
   (see [Ka]). Then
       \[K_n^M(k_n)/p\cong \bZ/p \quad  and \quad K_m^M(k_n)/p=0\quad 
       for\ m>n.\]

  \section{$Ah^{*,*'}(Y)=ABP^{*,*'}(Y;\bZ/p)$} 

In this section, we study the algebraic cobordism
for the field $k$ with  $K_{n+2}^M(k)/p=0$.
Moreover we assume that $k$ contains a $p^2$-th root of the unity.
Throughout this section,
we assume the above $k,p$, that is,
\[ K_{n+2}^M(k)/p=0,\quad Q_0(\tau)=\zeta_p =0\in K_1^M(k)/p.\].

Let $Y$ be a smooth algebraic variety over $k$.
 We study the $BP$ version
of the algebraic cobordism
\[ ABP^{*,*'}(Y)=MGL^{*,*'}(Y)\otimes _{MU^*}BP^{*},
\]
\[ (\ MGL^{*,*'}(Y)_{(p)}\cong ABP^{*,*'}(Y)\otimes_{BP^*}MU_{(p)}^*,\ )\]
\[\Omega^*(Y)=MGL^{2*,*}(Y)\otimes _{MU^*}BP^*\]
 defined by Voevodsky and Levine-Morel
([Vo1], [Le-Mo1,2], [Ya1,4]), with the coefficient ring
$BP^*=\bZ_{(p)}[v_1,v_2,...]$,
where $|v_n|=-2(p^n-1)$. 

The Chow ring $CH^*(Y)$ is written as a quotient of $\Omega^*(Y)$
 ([Vo1],[Le-Mo1,2]) that is,
\[ CH^*(Y)\cong \Omega^*(Y)\otimes _{BP^*}\bZ_{(p)}
\cong \Omega^*(Y)/(BP^{<0}).\]
   
   Let us write by $ABP^{*,*'}(Y;\bZ/p)$ the mod $p$
$ABP$-theory so that
\[\to  ABP^{*,*'}(Y)\stackrel{\times p}{\to}
ABP^{*,*'}(Y)\to ABP^{*,*'}(Y;\bZ/p)\to...\]
and $ABP^{2*,*}(Y;\bZ/p)\cong \Omega^{2*}(Y)/p$.
We simply write 
\[ Ah^{*,*'}(Y)=ABP^{*,*'}(Y;\bZ/p).\]
\begin{lemma}
If $n+1\le 2p-1$, then $Ah^{*,*'}(pt.)\cong H^{*,*'}\otimes BP^{*}$
where $pt.=Spec(k)$ and $H^{*,*'}=H^{*,*'}(pt.;\bZ/p)$.
\end{lemma}
\begin{proof}  Consider the Atiyah-Hirzebruch spectral sequence
\[ E_2^{*,*',*''}\cong H^{*,*'}(pt.;\bZ/p)\otimes
BP^{*''}\Longrightarrow Ah^{*,*'}(pt).\]
It is known the first (possible) non zero differential is
\[ d_{2p-1}(x)=v_1\otimes Q_1(x).\]

Here $H^{*,*'}=0$ for $*>n+1$ from the assumption for $k$.
Recall $H^{0,i}(pt.;\bZ/p)$ is generated by $\tau^i$.
When $p$ is odd, $Q_i$ is a derivation. and we see
$Q_1(\tau^i)=0$ since $Q_1(\tau)=0$. For $p=2$,
we have $Q_1(\tau^2)=\rho^3$ (Lemma 6.3 in [Ya1]).
Since $\rho=-1=0\in K_1^M(k)/2$, we see $d_3=0$.
Hence assumption of this lemma, this spectral sequence
collapses.
\end{proof}

Hereafter this section we assume that
$k$ is a field $k_*^M(k)/p\cong 0$ for $*>n+1$ and
$2p-1>n+1$ and $\zeta_{p^2}\in k$.
Therefore $h^{*,*'}(X)$ is an $H^{*,*'}\otimes BP^*$-algebra in this case.

{\bf Remark.} When $p=2$ and $k=\bR$, the algebra
$Ah^{*,*'}(Spec(\bR))$ is quite complicated
([Ya1]).

For smooth $Y_1,Y_2$ and $\theta\in Ah^{2*,*}(Y_1\times Y_2)$ for $*=dim(Y_2)$, we can define 
\[ f_{\theta}: Ah^{*,*'}(Y_1)\to h^{*,*'}(Y_2)\quad by\
f_{\theta}(x)=pr_{2*}(pr_1^*(x)\cup \theta) \]
where $pr_i : Y_1\times Y_2\to Y_i$ is the $i$-th projection.
For the projector $p$ (i.e., $p=f_{\theta}$ with $p\cdot p=p$),
we can define the motive $M=(Y,p)$ with
$Ah^{*,*'}(M)=p\cdot Ah^{*,*'}(Y)$.  Thus we have the category
$Ah$-motives.
\begin{lemma}
For smooth $Y$ and a motive $M\subset M(Y)$, we have 
the Atiyah-Hirzebruch spectral sequence
\[ E_2^{*,*',*''}\cong H^{*,*'}(M;\bZ/p)\otimes BP^*
\Longrightarrow Ah^{*,*'}(M).\]
\end{lemma}
\begin{proof}
(See Lemma 7.1 in [Ya4].)
We only need to see that $d_rf_{\theta}=f_{\theta}d_r$
for the differential $d_r$ in the spectral sequence.
For $\theta\in CH^*(Y\times Y)/p\subset Ah^{2*,*}(Y\times Y)$, we have
\[ f_{\theta}(d_r(x))=pr_{2*}(pr_1^*(d_r(x)\cdot \theta)
=pr_{2*}(d_r(pr^*_1(x)\cdot \theta)).\]
The last equation follows $d_r(\theta)=0$, since $w(\theta)=0$.

We know (e.g. $\S5$ in [Ya])
\[ pr_{2*}(x)=i^*(Th_Y(1)\cdot x)\]
where $Th_Y(1)\in H^{2m,m}(Th_Y(V);\bZ/p)$
is the Thom class for some
bundle $V$ over $Y$ and $i:\bT^m\times Y\subset Th_Y(V)\times Y$.  Since $w(Th_Y(1))=0$,
we see
$ d_r(Th_Y(1)\cdot x)=Th_Y(1)\cdot d_r( x).$
Therefore $pr_{2*}$ commutes with $d_r$.  
  Hence so does $f_{\theta}$. 
\end{proof}

For the Rost motive $R_n$, we will study the 
Atiyah-Hirzebruch spectral sequence
\[ (*)\quad E_2^{*,*',*''}\cong
H^{*,*'}(R_n:\bZ/p)\otimes BP^*\Longrightarrow Ah^{*,*'}(R_n).\]
From Theorem 3.3, we still have
\[ H^{*,*'}(R_n;\bZ/p)\cong
          H^{*,*'}\oplus 
             (\tilde  Q(n-1)\{a'\}\oplus \tilde H^{*,*'}\{t\})
\{1,t,...,t^{p-2}\}\]
where  $\tilde Q(n-1)=Q(n-1)-\bZ/p\{Q_0...Q_{n-1}\}$.
          \begin{lemma}  The nonzero differential $d_r$ in AHss $ (*)$
has the form
\[ d_{2p^i-1}(x)=v_i\otimes Q_i(x)\quad mod (I_i)\quad with\ I_i=(v_1,...,v_{i-1})).\]
\end{lemma}
The proof of this lemma will give just before Corollary 4.8 below. 
\begin{lemma}
Let $c_i=Q_0...\hat Q_i...Q_{n-1}(a')$ and 
\[ A=\oplus BP^*/I_n\{c_0\}\oplus \bigoplus
_{i=1}^{n-1}BP^*/I_{i}\{c_i\}. \]
Then the infinity term  in the AHss $(*)$ is written as
\[ E_{\infty}^{*,*',*''}\cong
          H^{*,*'}\otimes BP^*\oplus 
             (A\oplus \tilde H^{*,*'}\otimes BP^*/p\{t\})
\{1,t,...,t^{p-2}\}.\]
  \end{lemma}
\begin{proof}
Let us write $Q(i)'=\Lambda(Q_i,...,Q_1)$.
For $x\in\tilde Q(n-1)$, considering that $Q_0$ is contained in $x$ or not,
we have the decomposition
\[\tilde Q(n-1)\cong Q(n-1)'\oplus \tilde Q(n-1)'Q_0.\] 
 Considering $Q_{n-1}$ is contained in $x$ or not  
for $x\in \tilde Q(n-1)'$, we have
\[  \tilde Q(n-1)'Q_0\cong Q(n-2)'Q_0\oplus \tilde Q(n-2)'Q_{n-1}Q_0.\]
Continuing this argument, we can see that
\[ \tilde Q(n-1)\cong Q(n-1)'\oplus \bigoplus_{i=1}^{n-1}
Q(i-1)'Q_{n-1}...Q_{i+1}Q_0 \]

For a $BP^*\otimes Q(n-1)'$-module $B$, let us write
$E_0(B)=B$ and $E_{r+1}(B)=H(E_r(B);v_r\otimes Q_r)$.
Then we easily see by induction on $j$
\[ E_j(BP^*\otimes Q(i)')\cong 
\begin{cases} BP^*/I_{j+1}\otimes\Lambda(Q_i,...,Q_{j+1})\{Q_j...Q_1\}
\quad j<i \\
   BP^*/I_{i+1}\{Q_i...Q_1\}\quad otherwise.
\end{cases}
\]
Hence we see that 
\[E_{n-1}(BP^*\otimes Q(i-1)'Q_{n-1}...Q_{i+1}Q_0(a'))\]
\[ \cong BP^*/I_{i}\{Q_{n-1}...Q_{i+1}Q_{i-1}...Q_1Q_0(a')\}\]
\[\cong BP^*/I_{i}\{
Q_{n-1}....\hat Q_{i}...Q_0(a')\}\cong BP^*/I_{i}\{c_{i}\}.\]

Therefore we get 
\[ E_{n-1}(BP^*\otimes \tilde Q(n-1)'(a'))
\cong A=\oplus BP^*/I_n\{c_0\}\oplus \bigoplus
_{i=1}^{n-1}BP^*/I_{i}\{c_i\}. \]

From Lemma 4.3, all differential are of the forms $
d_{2p^i-1}=v_i\otimes Q_i$ (in particular $d_r(BP^*\otimes \tilde H^{*,*'}y^j)=0$), we have this lemma.
\end{proof}

We recall the following lemma to see the relations
between $c_i$ in $Ah^{*,*'}(R_n)$.
\begin{lemma} (Corollary 3.4 in [Ya4])
Let $x\in CH^*(Y)\cong E_{\infty}^{2*,*,0}$ and $v_sx=0$ in $E_{\infty}^{2*,*,*'}$ for the
AHss converging to $ABP^{*,*'}(Y)$.  
Then there is $b\in H^{*,*'}(Y;\bZ/p)$ with
$Q_s(b)=x$  and  a relation in $ABP^{2*,*}(Y)$ 
\[ v_sx+v_{s+1}x_{s+1}+...+v_kx_k+...=0\quad mod(I_{\infty}^2)\]
with $x_k=Q_k(b)$ in $H^{2*,*}(Y;\bZ/p)$ for all $k\ge s$.
\end{lemma}
\begin{cor} In $ABP^{2*,*}(R_n)$, we have 
$v_ic_j=v_jc_i$ $mod(I_{\infty}^2)$. The restriction map
$Res(c_i)=v_iy\ mod(I_{\infty}^2)$ in $ABP^{2*,*}(\bar R_n)$.
\end{cor}
\begin{proof}  First note that
$c_i$ exists in (the integral) $ABP^{2*,*}(R_n)$, since $c_i$
exists in $H^{2*,*}(R_2)$ (considering the integral AHss).
Let $i>0$.  Since $pc_{i}=0$ in the spectral sequence
there is $x\in H^{*,*'}(R_2)$ such that
$Q_0x=c_i$ and there is a relation
\[ pc_i+v_1Q_1(x)+...+v_{i}Q_i(x)=0\quad mod(I_{\infty}^2).\]
By the dimensional reason, we see $x=Q_{n-1}...\hat Q_i..Q_1(a')$.  Hence the above equation is written as
\[ pc_i-v_ic_0=0\quad mod(I_{\infty}^2)\]
since $Q_ix=Q_1...Q_{n-1}(a')=c_0$.

 From lemma 3.11, we know $Res(c_0)=py$ $mod(p^2,BP^{<0})$.
So $Res(c_i)=v_iy$ since $ABP^{*,*'}(\bar R_n)$ is $BP^*$-free.  The first relation is proved similarly
by using $v_{j}c_i=0$ $j<i$ in the spectral sequence.
\end{proof}
\begin{thm}
We have an isomorphism 
\[ Ah^{*,*'}(R_n)\cong
          H^{*,*'}\otimes BP^*\oplus 
             (BP^*/I_n\{c_0\}\oplus 
I_{n}\{t\}\oplus BP^*\otimes \tilde H^{*,*'}\{t\})
\{1,t,...,t^{p-2}\}\]
\end{thm}
\begin{proof}
We have an isomorphism
\[ gr(Res):\bigoplus_{i=1}^{n-1}BP^*/I_i\{c_i\}
\to \bigoplus_{i=1}^{n-1}BP^*/I_i\{v_it\}\cong gr(I_{n}t).\]
Therefore we get the result
from Lemma 4.4 and Lemma 4.6.
\end{proof}

\begin{proof}[ Proof of Lemma 4.3]
For an element $x=Q_i...Q_1Q_0(a')$ in $E_{2p^i}$, the next nonzero
differential is $d_{2p^{i+1}-1}$. Otherwise
\[ d_r(x)=v\otimes x' \quad |v_i|>|v|>|v_{i+1}|.\]
But by dimensional reason, there does not exist
nonzero $x$ such that $|x|<|x'|<|x|+|Q_{i+1}|$.

The elements $\tilde H^{*,*'}(t^j)$ is generated
as a $BP^*\otimes H^{*,*'}$-algebra  by $\tau t^j$ and 
$bt^j$ for $b\in K_1^M(k)/p$.  Since 
$w(bt^j)=1$, the differential image $w(d_r(bt^i))=0$.
Hence $d_r(bt^j)\in \oplus_{i,k} BP^*/I_i\{c_it^k\}$.
But if $d_r(bt^j)\not =0$, then it contradicts to
the existence of the restriction map $gr(Res)$ in the proof of Theorem 4.7.

At last we consider $\tau t$.
Since $\tau b_1t^i=0$ in the spectral sequence,
$\tau b_1t^i=vc$ for $v\in BP^{<0}$ in $Ah^{*,*'}(R_n)$.
Restrict it to $Ah^{*,*'}(\bar R_n)$,  we have $\tau
v_1t^{i+1}=v(c|_{\bar k})$.  Since $Ah^{*,*'}(\bar R_n)$
is $BP^*/p[\tau]$-free, we see $v=v_1$ and $c|_{\bar k}=\tau t^{i+1}$. This  means that $\tau t^{i+1}$ exists in $Ah^{*,*'}(R_n)$, and it is a permanent cycle.
\end{proof}
\begin{cor} 
The motivic cobordism $Ah^{*,*'}(R_n)$ is isomorphic
to the $Ah^{*,*'}$-submodule of 
$ Ah^{*,*'}[t]/(t^p)$ (for $Ah^{*,*'}\cong BP^*\otimes H^{*,*'}$) generated by 
\[1,\ \ c_0t^{j-1},\ \ I_nt^j, \ \ \tilde H^{*,*'}t^j \quad for \ 1\le j\le p-1.\] 
\end{cor}
We consider here the another cobordism theory
\[ Ak^{*,*'}(Y)=ABP\la n-1\ra^{*,*'}(Y;\bZ/p)\]
so that 
$Ak^{*,*'}\cong BP\la n-1\ra \otimes H^{*,*'}\cong \bZ/p[v_1,...,v_{n-1}]\otimes H^{*,*'}.$
In this case, we note $I_n\cong \tilde k^{*}=BP\la n-1\ra^*-\bZ/p\{1\}$.  Hence we can see

\begin{cor} 
The motivic cobordism $Ak^{*,*'}(R_n)$ is isomorphic
to the $Ak^{*,*'}$-submodule of 
$ Ak^{*,*'}[t]/(t^p)$ (for $Ak^{*,*'}\cong BP\la n-1\ra^*\otimes H^{*,*'}$) generated by 
\[\ \  1,\ \ c_0t^{j-1},\ \ \tilde Ak^{*,*'}t^j\qquad   for\ 1\le j\le p-1,\]
where $ \tilde Ak^{*,*'}=Ak^{*,*'}-\bZ/p.$ 
\end{cor}
In the next section, we give an another short proof for this fact but assuming some properties for $Ak$-motives.


\begin{prop}  Let $G$ be a group of type of $I$.
Let $X=\bG_k/B_k$ be a twisted flag variety.
Then there is a filtration such that  $grAh^{*,*'}(X)$ is isomorphic to the
$BP^*\otimes H^{*,*'}$-subalgebra generated by 
\[1,\ \  \tilde H^{*,*'}y^{i+1} \ (0\le i\le p-2),\ \
and \ \ S(t) \] 
in $BP^*\otimes H^{*,*'}[y]/(y^p)\otimes S(t)/(R_1,R_2'')$  where
\[ R_1=(b_ib_j,b_k |1\le i,j\le 2p-2<k\le \ell), \]
\[ R_2''=(\tilde H^{*,*'}{b_{2i-1}},\tilde H^{*,*'}b_{2i},
   v_1b_{2i}| 1\le i\le p-1).\]
\end{prop} 
Note that $\tau b_{2i-1}=0$ in AHss, but
$\tau b_{2i-1}=v_1\tau y$ in $Ak^{*,*'}(X)$.

Define the etale cobordism theory by
$ Ah^{*}_{et}(Y)=lim_{N}(\tau^N Ah^{*,*'}(Y)).$
\begin{cor} 
We have the isomorphism
\[ grAh^{*}_{et}(X)\cong K_*^M(k)/p\otimes BP^*
  [y]/(y^p)\otimes S(t)/(b).\]
\end{cor}

Recall that $AK(n)^{*,*'}(Y)$ be the motivic Morava
K-theory such that
\[ AK(n)^{2*,*}(pt.)\cong K(n)^*=\bZ/p[v_n, v_n^{-1}].\]
\begin{cor}  Suppose  the same assumption for
$X,k$ as Theorem 4.10.  Then we have
\[ grAK(1)^{*,*'}(X)\cong K(1)^*\otimes H^{*,*'}\otimes
S(t)/(R^K_1,R^K_2)\]
where $R^K_1=(b_{2i},b_k| 1\le i\le p-1,\ 2p-1\le k\le \ell)$ 
and
\[R_2^K=( v_1(v_1^{-1}b_1)^i-b_{2i-1},\ \ b_1^p).\]
\end{cor}
\begin{proof}
In $Ah^{*,*'}(X)$, we have $\tau b_1=\tau v_1y$. (Note
$\tau y$ exists in $Ah^{*,*'}(X)$, but $y$ itself does not.)
Let new $y=v_1^{-1}b_1$ in $AK(1)^{*,*'}(X)$. Then we have
$b_{2i-1}=v_1y^i=v_1(v_1^{-1}b_1)^i$ in $AK(1)^{*,*'}(X)$.
Of course it holds for this $y$ that   $hb_1=hv_1y$ for $h\in \tilde H^{*,*'}$.
\end{proof}

Recall 
$
grK(1)^*(G/T)\cong K(1)^*[y]/(y^p)\otimes S(t)/(b).$
Here $b_{2i-1}=v_1y^i$, in particular, $y=v_1^{-1}b_1$
in $K(1)^*(G/T)$.  Thefore we see
\[ K(1)^*(G/T)\cong K(1)^*\otimes S(t)/(R_1^k,R_2^K)
\cong AK(1)^{2*,*}(X).\]
Hence $AK(1)^{*,*'}(X)\cong H^{*,*'}\otimes K(1)^*(G/T)$.

The graded algebra associated to the gamma filtration
of the (topological $mod$ $K$-theory) $K^*(G/T;\bZ/p)$ is given ([Ya6])
\[ gr(1)_{\gamma}^*(G/T)\cong gr(1)^*_{\gamma}(X)\]
\[ \cong CH^*(X)/(p,b_{2i})\cong S(t)/(p,R_1^K,b_{2i-1}b_{2j-1}
|1\le i,j\le p-1).\]
The associated graded algebra in the above corollary
gives more strong information for the gamma filtration
than $gr(1)_{\gamma}^*(X)$. For example, from
$b_1^2=0\in gr(1)_{\gamma}^*(X)$, we see
$b_1^2\in F_{\gamma}^{2|b_1|+2}$, but from 
$b_1^2=v_1b_3\in gr(K(1)^*(X))$,
we see
\[b_{1}^2\in F_{\gamma}^{|b_3|}, \quad where \ 2|b_1|+2p-2=|b_3|.\]

  \section{ $Ak^{*,*'}(Y)=ABP\la n-1\ra^{*,*'}(Y;\bZ/p)$}
Throughout this section,
we assume 
\[ K_{n+2}^M(k)/p=0,\quad Q_0(\tau)=\zeta_p =0\in K_1^M(k)/p\]
as the preceding section.

In algebraic topology, $BP\la m\ra^*(Y)$ is the cohomology theory
with the coefficient ring $BP\la m\ra^*=\bZ_{(p)}[v_1,...,v_m]$.
Let us write
\[ Ak^{*,*'}(Y)=ABP\la n-1\ra^{*,*'}(Y;\bZ/p).\]
Hence $Ak^{*,*'}(pt)\cong H^{*,*'}[v_1,...,v_{n-1}]$.
In this section we compute $Ak^{*,*'}(R_n)$ 
assuming the existence of some category of
$Ak^{*.*'}$-$motives$. The result 
is get quite easily,
and  of course, coincides with Corollary 4.9.
For ease of notations, hereafter this section, we simply write
\[ c=c_0,\ \delta=\delta_a,\ \xi'=Q_n(c_0),\ \xi=\xi_a,\ \chi=\chi_a,\
M^i=M_a^i.\]

\begin{lemma} We have $ Ak^{*,*'}(\chi)\cong Ak^{*,*'}
\oplus Ak^{*,*'}(\tilde \chi)$, and 
\[ Ak^{*,*'}(\tilde \chi)\cong
\bZ/p[\xi]\{c,\delta,\xi',\xi\}\]
where  $c=Q_{n-1}...Q_1(a')$, $\delta=Q_0c$ and $\xi'=Q_n(c), Q_0(\xi')=\xi$.
\end{lemma}
\begin{proof}
We consider AHss
\[E_2^{*,*',*''}\cong H^{*,*'}(\tilde \chi;\bZ/p)
\otimes BP\la n-1\ra \Longrightarrow Ak^{*,*'}(\tilde \chi).\]
Here $H^{*,*'}(\tilde \chi; \bZ/p)\cong 
Q(n)\{a'\}\otimes \bZ/p[\xi]$.

We consider the decomposition
\[ Q(n)\cong \Lambda(Q_0,Q_n)\otimes Q(n-1)'\quad with\
Q(n-1)'=\Lambda(Q_{n-1},...,Q_1).\]
We easily see all differentials are of form $d_{2p^r-1}(x)=v_r\otimes Q_r(x)$.  Hence we have
\[ E_{2p^{n-1}-1}^{*,*',*''}\cong 
\Lambda(Q_0,Q_n)\otimes BP\la n-1\ra^*/I_n\{Q_{n-1}...Q_1(a')\}
[\xi]\]
\[ \cong \Lambda(Q_0,Q_n)\otimes \bZ/p\{c\}[\xi]
\cong \bZ/p\{c,Q_0c,Q_n(c),Q_0Q_n(c)\}[\xi]. \]
Here we used
$ BP\la n-1\ra^*/I_n\cong \bZ/p[v_1,...,v_{n-1}]/(v_1,...,v_{n-1})\cong \bZ/p.$
\end{proof}

Elements $1,\delta',\delta,\xi',\xi$ in $Ak^{*,*'}(\chi)$ are contained in  $D_i$ as follows 
\[\begin{CD}
    @>{D_{-1}}>{1}>   @>{D_{0}}>{c,\delta}> @>{D_{1},...}>{0,...}> @>{D_{p-2}}>{0}> @>{D_{p-1}}>{\xi',\xi}> @>{D_{p}}>{c\xi,\delta\xi}>
\end{CD}  \]

Recall the category
$Ak$-motives in the preceding section.
More strongly, we assume
\begin{as}  There is a triangular category $\bD\bM(Ak)$
which contains
the category of $Ak$-motives 
such that for $X\in \bD\bM(Ak)$, we can define $Ak^{*,*'}(X)$,
and 
for a cofibering $X\to Y\to Z$ in $\bD\bM(Ak)$, we have
the induced exact sequence
\[ ...\to Ak^{*,*'}(Z)\to Ak^{*,*'}(Y)\to Ak^{*,*'}(X)\to Ak^{*+1,*'}(Z)\to..\]
Moreover,  $\chi$ is an object of $\bD\bM(Ak)$
and 
\[  Hom_{\bD\bM(Ak)}(M^i,\chi(*')[*])\cong
     Ak^{*,*'}(M^i).\]
 \end{as}

From the assumption we can define
the map $\delta: \chi\to \chi(*')[*]$
in the category $\bD\bM(Ak)$. 
We still know $\delta Q_n(a')=a'\xi$, which induces
$\delta \xi'=c\xi$.

The map $\times \delta : Ak^{*,*'}(\chi;\bZ/p)\to
Ak^{*,*'}(\chi;\bZ/p)$ is written as
\[ \begin{CD}
   @. @>{D_{-1}}>{\ \ 1}>  
 @>{D_{0}}>{c,\delta}>
 @>{D_{1},...,D_{p-3}}>{0,...,0}>
 @>{D_{p-2}}>{0}> 
@>{D_{p-1}}>{\xi',\ \xi}> 
@>{D_{p}}>{c\xi,\ \delta\xi}> \\
 @. @V{\delta}VV  @V{\delta}VV   @V{\delta}VV  @V{\delta}VV @V{\delta}VV  @V{\delta}VV \\
  @>{1}>{D_{-1}}> 
  @>{c,\delta}>{D_0}>
 @>{0,0}>{D_{1}}>
@>{0,...,0}>{D_2,...,D_{p-2}}> 
@>{\xi',\ \xi}>{D_{p-1}}>
 @>{c\xi,\ \delta\xi}>{D_{p}}>
@>{0,\ 0}>D_{p+1}>
\end{CD}  \]

 We ca easily see
$ Coker(\delta)\cong 
Ak^{*,*'}\oplus \bZ/p\{c\}\oplus
\bZ/p[\xi]\{\xi',\xi\},$  and
\[
\quad  Ker (\delta)\{t\}\cong 
\tilde Ak^{*,*'}\{t\}
\oplus \bZ/p[\xi]
\{ct,\delta t\}.\]
\begin{lemma}
We have the isomorphism 
\[ Ak^{*,*'}(M;\bZ/p)\cong 
Ak^{*,*'}\oplus \bZ/p\{c\}\oplus \tilde Ak^{*,*'}\{t\} \oplus 
  \bZ/p\{ct,\delta t, \xi',\xi\}\otimes\bZ/p[\xi].\]
 \end{lemma} 

Next, we compute $Ak^{*,*}(M^i;\bZ/p)$ for $1<i<p-1$.
As lemma 3.9, we can compute  
\begin{lemma}
We have the isomorphism 
\[ Ak^{*,*'}(M^{p-2};\bZ/p)\cong 
Ak^{*,*'}\oplus (\bZ/p\{c\}\oplus \tilde Ak^{*,*'}\{t\})
\{1,t,...,t^{p-3}\}\]
\[\oplus (\bZ/p\{ct^{p-2},\delta t^{p-2},
\xi',\xi\})\otimes \bZ/p[\xi].\] \end{lemma}

At last we compute $Ak^{*,*'}(M^{p-1})$.
Recall Lemma 3.9 and the map $r_{p-1}^*$ is given in $H^{*,*'}(M^{p-2};\bZ/p)$
by

\[ 1\mapsto \delta t^{p-2},\quad 
 a'\mapsto Q_n(a'), \quad Q_n(a')\mapsto a'\xi t^{p-2}. \] 
Hence, in $Ak^{*,*'}(M^{p-2})$,  we have 
$1\mapsto \delta t^{p-2}$ \ $ c\mapsto \xi'$,\ \
\ \ $\delta \mapsto \xi$,\ \ 
$\xi'\mapsto c\xi,$ \ \ $\xi\mapsto \delta t^{p-2}\xi$.
The map $r_{p-1}^*$ is given as follows.
\[\begin{CD}
  @. @.  @>{D_{-1}}>{\ Ak^{*,*'}\{1\}}>  
 @>{D_{0}}>{c,\delta}>
 @>{D_{i+1}}>{0,0}>
@>{D_{p-1}}>{\xi',\xi}>
 @>{D_{p}}>{c\xi,\delta\xi}> \\
 @. @. 
@V{r_{p-1}^*}VV 
@V{r_{p-1}^*}VV 
@V{r_{p-1}^*}VV
@V{r_{p-1}^*}VV 
@V{r_{p-1}^*}VV\\
  @>{Ak^{*,*'}}>{D_{-1}}> 
  @>{ct^i, \tilde Ak^{*,*'}t^{i+1}}>{D_i}>
 @>{c't^{p-2},\delta t^{p-2}}>{D_{p-2}}>
@>{\xi',\xi}>{D_{p-1}}>
 @>{0,0}>{D_{p+i}}>
@>{c\xi t^{p-2}, \delta \xi t^{p-2}}>{D_{2p-2}}>
@>{\xi'\xi,\xi^2}>{D_{2p-1}}>
\end{CD}  \]
Note that $Ker(r_{p-1}^*)\{t\}\cong \tilde Ak^{*,*'}t^{p-1}$.
Using these, we see the following theorem with Assumption 5.2, while it also follows from Corollary 4.9 (without Assumption 5.2).
\begin{thm}   We have the isomorphism
\[ Ak^{*,*'}(R_n)\cong Ak^{*,*'}\oplus
(\bZ/p\{c_0\}\oplus \tilde Ak^{*,*'}\{t\})\{1,...,t^{p-2}\}.\]
\end{thm}

{\bf  Remark.} 
When $Ak^{*,*'}(X)=ABP\la n\ra^{*,*'}(X)$, we see that
 $\delta_a$ does not exist  in $Ak^{*,*'}(\chi)$.  Hence we can not
do the above arguments.

\section{motivic cohomology over $\bR$}

When $p=2$, important cases have the property
with $K_*^M(k)/2\not =0$ for all $*\ge 0$.
Let $\bR$ be the field of real numbers.
Let $\rho=-1\in K_1^M(\bR)/2\cong \bR^{\times}/(\bR^{\times})^2.$  Then it is well known  
\[K_*^M(\bR)/2\cong H_{et}^*(Spec(\bR);\bZ/2)\cong\bZ/2[\rho].\]
Hence the motivic cohomology is given as
\[ H^{*,*'}(Spec(\bR);\bZ/2)\cong
K_*^M(\bR)\otimes \bZ/2[\tau]\cong \bZ/2[\rho,\tau]\]
where $0\not =\tau\in H^{0,1}(Spec(\bR))/2.$

Let $a=\rho^{n+1}\in K_{n+1}^M(\bR)/2$.  Since $K_*^M(\bR)/Ker(a)\cong \bZ/2[\rho]$, we have
from Theorem 3.1
\[ H^*(\chi_a;\bZ/2)\cong \bZ/2[\rho]\otimes (\bZ/2[\tau]\oplus Q(n)(a')\otimes \bZ/2[\xi])
\quad (*)\]
where $\xi=\delta_a^2$.  Note $H_{et}^*(\chi_a;\bZ/2)
\cong H_{et}^*(\bR;\bZ/2)\cong \bZ/2[\rho]$.
\begin{lemma}
Let $Q^{\ep}=Q_0^{\ep_0}...Q_{n-1}^{\ep_{n-1}}$ for
$\ep=(\ep_0,...,\ep_{n-1})$ with $\ep_i=0$ or $1$.
Then $H^*(\chi_a;\bZ/2)[\tau,\tau^{-1}]\cong \bZ/2[\rho,\tau,\tau^{-1}]$, and we have
\[ Q^{\ep}(a')=\tau^{-d(\ep)-1}\rho^{f(\ep)+n+1}\]
where $d(\ep)=\sum_i\ep_i2^i$ and $f(\ep)=\sum_i \ep_i(2^{i+1}-1)$.
\end{lemma}\begin{proof}
By definition, we see $Q_0\tau=\rho$. The $Q_0$ action is a derivative and we have
\[0=Q_0(1)=Q_0(\tau\tau^{-1})=\rho\tau^{-1}+\tau Q_0(\tau^{-1}),\]
hence $Q_0(\tau^{-1})=\rho\tau^{-2}$.
We have $Q_1(\tau^{-1})=\rho^3\tau^{-3}$
by using that the coproduct $\psi$ is given (Proposition 13.4 in [Vo])
\[ \psi(Q_1)=Q_1\otimes 1+1\otimes Q_1+\rho Q_0\otimes Q_0.\]
Similarly, we can prove the lemma (for details see proof of
Lemma 4.6 in [Ya]).
\end{proof}

\begin{lemma}  The motivic cohomology
$H^{*,*'}(\chi_a;\bZ/2)$ has no $\tau$-torsion elements
and, it is the $\bZ/2[\rho,\tau]$-submodule
of $\bZ/2[\rho,\tau,\tau^{-1}]$ generated by
\[ 1,\ \  Q^{\ep}(a')=\tau^{-d(\ep)-1}\rho^{f(\ep)+n+1}.\]
\end{lemma}
\begin{proof}
We can define a $Q(n-1)$-module map $j:H^*(\chi_a;\bZ/2)\to \bZ/2[\rho,\tau,\tau^{-1}]$ by $j(a')=\tau^{-1}\rho^{n+1}$ and so
$j(Q^{\ep}(a'))=\tau^{-d(\ep)-1}
\rho^{f(\ep)+n+1}$.  This map is injective
from $(*)$ and $Q(n)=\Lambda(Q_n)\otimes Q(n-1)$
is $Q(n-1)$-free.
In fact $Q^{\ep}(a')$ generate free $\bZ/2[\rho]$-modules
in $\bZ/2[\rho,\tau,\tau^{-1}]$.
\end{proof}

\begin{lemma}
We have $\delta a'=Q_n(a')$ and $\delta Q_n(a')=a'\xi$.
\end{lemma}
\begin{proof}
The first equation follows from the fact that
$H^{*,*'}(\tilde \chi;\bZ/2)$ is a $Q(n-1)$-free and
\[ Q_{n-1}...Q_0(a'\delta-Q_n(a'))=\delta^2-\xi=0.\]
Then $\delta Q_n(a')=\delta^2a'=\xi a'$.
\end{proof}
The map $\times \delta : H^{*,*'}(\chi;\bZ/2)\to
H^{*,*'}(\chi;\bZ/2)$ is written as
\[ \begin{CD}
   @. @>{D_{-1}}>{\ \ 1}>  
 @>{D_{0}}>{a',\delta}>
@>{D_{1}}>{Q_n(a'),\ \xi}> 
@>{D_{2}}>{a'\xi,\ \delta\xi}> \\
 @. @V{\delta}VV  @V{\delta}VV   @V{\delta}VV  @V{\delta}VV
@V{\delta}VV \\
  @>{1}>{D_{-1}}> 
  @>{a',\delta}>{D_0}>
@>{Q_n(a'),\ \xi}>{D_{1}}>
 @>{a'\xi,\ \delta\xi}>{D_{2}}>
`@>{Q_n(a')\xi, \ \xi^2}>{D_3}>
\end{CD}  \]

\begin{thm}  (Theorem 1.3 in [Ya2])   The motivic cohomology
$H^{*,*'}(R_n;\bZ/2)$ is isomorphic to the $\bZ/2[\rho,\tau]$-subalgebra of 
\[ H^{*}_{et}(\bR;\bZ/2)[\tau,\tau^{-1}]/(\rho^{2^{n+1}-1})\cong \bZ/2[\rho,\tau,\tau^{-1}]/(\rho^{2^{n+1}-1})\]
 generated by $1,\ \ Q^{\ep}(a')=\tau^{d(\ep)-1}\rho^{f(\ep)+n+1}$.
\end{thm}

\begin{cor} 
We have the $\bZ/2[\rho,\tau]$-module isomorphism
\[  H^{*,*'}(R_2;\bZ/2)\cong 
(\bZ/2[\tau]\oplus \bZ/2\{a',Q_0a',Q_1a'\})\otimes
\bZ/2[\rho])/(\rho^7)\]
 \[   \subset  H_{et}^*(R_2;\bZ/2)[\tau,\tau^{-1}]\cong
\bZ/2[\rho,\tau,\tau^{-1}]/(\rho^7) \]
identifying $a'=\tau^{-1}\rho^3$, $Q_0a'=\tau^{-2}\rho^4$,
$Q_1a'=\tau^{-3}\rho^6$.
\end{cor}
\begin{cor}  We have a $\bZ/2[\tau]$-module isomorphism
\[ H^{*,*'}(R_2;\bZ/2)\cong
\bZ/2[\tau]\{1,\rho,\rho^2,a',Q_0(a'),\rho Q_0(a'),Q_1(a')\}.\]
\end{cor}

The following theorem is shown for $G=G_2$ in
Corollary 5.5 in [Ya3].

\begin{thm}
Let $X=\bG_k/B_k$ for $G$ of type $(I)$.
Then we have the isomorphism
for $*\ge *'$
\[H^{*,*'}(X; \bZ/2)\cong
\bZ/2[\rho]\otimes  S(t)\{1,a'\}/(R_1,R_2,R_3)\]
where $R_1=(b_ib_j, b_k|1\le i,j\le 2 <k<\ell)$,
$R_2=(\rho^7, \rho^3b_1, \rho b_2)$ and \[R_3=((a')^2-\rho^2b_1, \rho^4a', b_sa'| 1\le s\le \ell).\]
\end{thm}
\begin{proof}
In $H^{*,*'}(X;\bZ/2)$, the relation $R_1$ also holds
since so does in $CH^*(X)/2\cong H^{2*,*}(X;\bZ/2)$.
 By dimensional reason, we can take 
$b_1=Q_0a'+\lambda t $ for $0\not =t\in S(t)/(b)$ and $\lambda\in \bZ$.  Since $ b_1|_{\bC}=Q_0a'|_{\bC}=0$
but $t|_{\bC}\not=0$, we have $\lambda=0\in \bZ/2$.
Similarly we can take $Q_1a'=b_2$.  Hence $R_2$ also
holds in $H^{*,*'}(X;\bZ/2)$.  We can take  $a'$ satifying $R_3$ from $H^{*,*'}(R_2;\bZ/2)$.
Hence we have the $\bZ/2[\rho]$-algebra map
\[A=\bZ/2[\rho]\otimes S(t)/(R_1,R_2,R_3)
\to H^{*,*'}(X;\bZ/2)\]
When $*\ge *'$, this map is additively isomorphic from Petrov-Semenov-Zainoulline theorem.  So is isomorphic as
$\bZ/2[\rho]$-algebra map.
\end{proof}
\begin{cor}  We have the $\bZ/2[\rho]$-algebra isomorphism
\[H^*_{et}(X;\bZ/2)\cong \bZ/2[\rho]\otimes S(t)/(R_1,R_2')\]
where $R_2'=(\rho^7, \rho^4=b_1,\rho^6=b_2)$.
\end{cor}

Consider the Borel spectral sequence
\[E_2^{*,*'}\cong H^*(B\bZ/2; H_{et}^{*'}(G_k/B_k;\bZ/2))
\Longrightarrow H_{et}^*(\bG_k/B_k;\bZ/2).\]
Here $H^*(B\bZ/2;\bZ/2)\cong \bZ/2[\rho]$
and $H^*_{et}(G_k/T_k;\bZ/2)\cong
\bZ/2[y]/(y^2)\otimes S(t)/(b)$.  The element
$y$ is not a permanent cycle of this spectral sequence.
We can see $d_7(y)=\rho^7$ and 
\[ E_8^{*,*'}\cong \bZ/2[\rho]/(\rho^7)\otimes
S(t)/(b).\]
This term becomes $E_{\infty}$-term and 
$grH^*_{et}(\bG_k/B_k;\bZ/2)\cong E_8^{*,*'}$.
Of course, this fact coincide
the above corollary by $b_1=\rho^4$ and $b_2=\rho^6$
in $H^*_{et}(X;\bZ/2)$.

\section{Witt groups.}

For a smooth variety $Y$ over a field $k$ with $1/2\in k$,
let $W(Y)$ 
denote the Witt group of $Y$.  Balmer defined the periodic Witt group
$W^i(Y)\cong W^{i+4}(Y),\ (i\in \bZ)$ 
with $W^0(X)=W(Y)$  (see for details, [Ba-Wa]).

Balmer and Walter [Ba-Wa] define the Gersten-Witt complex
\[0\to W(k(Y))\to \oplus _{x\in Y^{(1)}}W(k(x))\to ... 
\oplus _{x\in Y^{(n)}}W(k(x))\to  0.\]
Let $H^*(W(Y))$ denote the cohomology group of the above cochain complex, 
with $W(k(Y))$ places in degree $0$.    From the above complex, we have the
(Balmer-Walter) spectral sequence
\[E(BW)^{r,s}_2\cong \begin{cases}  H^r(W(Y))\ (s=4s')\quad 
\Longrightarrow  W^{r+s}(Y)\\
0\ \ (s\not =0(mod(4))
\end{cases}.\]

By the affirmative answer of the Milnor conjecture of quadratic forms
by Orlov-Vishik-Voevodsky [Or-Vi-Vo], we have the isomorphism of graded rings
$H^*(k(x);\bZ/2)\cong grW^*(k(x))$.  
Using this fact, Pardon and Gille ([Pa],[Gi],[To]) defined the spectral sequence
\[ E(GP)_2^{r,s}\cong H^r_{Zar}(Y;H_{\bZ/2}^s)
 \Longrightarrow H^r(W(Y))\cong E(BW)_2^{r,4s}\]
so that the differential $d_r$ has degree $(1,r-1)$ for $r\ge 2$.
Here $H_{\bZ/2}^s$ is the Zariski sheaf induced from the
  presheaf $H_{et}^s(V;\bZ/2)$ for open subset $V$ of $Y$.
  
  The above sheaf cohomology
   $H^r_{Zar}(Y;H_{\bZ/2}^s)$ relates the motivic cohomology $H^{*,*'}(Y;\bZ/2)$.
      Indeed, we get the long exact sequence from the 
solution of Beilinson-Lichtenbaum conjecture
 by Voevodsky [Or-Vi-Vo] 
  \[ ...\to H^{m,n-1}(Y;\bZ/2)\stackrel{\times \tau}{\to}
  H^{m,n}(Y;\bZ/2)\qquad \ \qquad \]
  \[\qquad \qquad \to
H_{Zar}^{m-n}(Y;H_{\bZ/2}^n)\to H^{m+1,n-1}(Y;\bZ/2)\stackrel{\times \tau}{\to}...\]
 Therefore, we have 
\begin{lemma}
$ E(GP)_2^{m-n,n}\cong H_{Zar}^{m-n}(Y;H_{\bZ/2}^{n})\cong$
\[ H^{m,n}(Y;\bZ/2)/(\tau)\oplus 
Ker(\tau)|H^{m+1,n-1}(Y;\bZ/2).\]
\end{lemma} 
In particular, 
$E(GP)_2^{m,m}\cong H^{2m,m}(Y;\bZ/2)\cong CH^m(Y)/2$,
\\  and $E(GP)_2^{m,m+1}\cong H^{2m+1,m+1}(Y;\bZ/2)$.
Moreover Totaro proved   
\begin{lemma}(Totaro [To])
If $x\in E(GP)_2^{*,*}$, then $d_2(x)=Sq^2(x)$. 
\end{lemma}
Many examples ([Ya2,5]
satisfy the following assumption.
\begin{as}
If $x\in E(GP)_2^{*,*+1}$,
then $d_2(x)=Sq^2(x)$.
\end{as}

It is known that the Witt group is written as the motivic 
Hermitian $K$-theory by Schlichting and Tripathi [Sc-Tr], namely, $
W^i(Y)\cong KO^{i+*,*}(Y)$  for $\ i>*.$
Hence $W^*(X)$ has the multiplicative structure.
In particular, we can see [Ya5] that the differentials in the both spectral sequences $E(GP)_r$ and $E(BW)_r$ are derivations.

Now we consider the cases $Y=X=\bG/B_k$.
At first, we recall the case $\bar X$.
Let us write $H(Y;Sq^2)=Ker(d)/Im(d)$  the homology with the differential
$d=Sq^2$ on $H^*(Y;\bZ/2)$.  We give dgree of $x\in H(Y;Sq^2)$ by half of the degree $|x|\in H^*(Y;\bZ/2)$.
\begin{thm} ([Ya5], [Zi])
There is an isomorphism
\[W^*(G_k)\cong W^*(\bar X)\cong H(G/T;Sq^2)\]
\[\cong H(\Lambda(y)\otimes S(t)/(b):Sq^2)\cong \Lambda(y,z_2,...,z_m)\]
where $deg(y)=1/2|y|=3$ and all $deg(z_i)$ are odd.
\end{thm}

It is known (Theorem 5.11 in [Ya2]) that there is an open
variety $U_a$ in some quadric $Q$ such that
$R_a\subset M(Q)$  and 
$H^{*,*'}(U_a;\bZ/2)\cong H^{*,*'}(R_a;\bZ/2).$
So we write by $W^*(R_a)$ the Witt group $W^*(U_a)$.
\begin{lemma}  If Assumption 7.3 is satisfied, then
$grW^*(R_2)\cong \bZ/2\{1,\rho,\rho^2\}$.  
\end{lemma}

\begin{prop}
Let $G$ be of type $(I)$ and $X=\bG_k/B_k$ for $k=\bR$.
Moreover Assumption 7.3 holds.  Then we have
\[grW^*(X)\cong grW^*(R_2)\otimes H(S(t)/(b);Sq^2)\]
\[ \cong  \bZ/2\{1,\rho,\rho^2\}\otimes
 \Lambda(z_2,...,z_m).\]
\end{prop}

To prove above lemma and proposition,
we need some lemmas.
First recall that
\[ grH^{*,*'}(X;\bZ/2)\cong  H^{*,*'}(R_2;\bZ/2)\otimes S(t)/(b)\quad (*)\]
where $H^{*,*'}(R_2;\bZ/2)\cong \bZ/2[\tau]\{1,\rho, \rho^2, a',Q_0(a'),\rho Q_0(a'), Q_1(a')\}$.
(Here note $\tau^{-1}\rho^3=a', Q_0(a')=b_1, Q_1(a')=b_2$ in $H^{*,*'}(X;\bZ/2)$.
Hence 
\[E(PG)_2\cong H^{*,*'}(X;H_{\bZ/2}^{*'})\cong grH^{*,*'}(X;\bZ/2)/(\tau).\]
Here in $H^{*}_{Zar}(X;H^{*'}_{\bZ/2})$, degree is given 
 $deg(a')=(1,2), deg(\rho Q_0(a'))=(2,3)$,  and 
degree of $Q_i(a')$ and 
elements in $S(t)/(b)$ are $(*,*)$.
Note that all elements in $E(GP)_2$ have degree
$(*+1,*)$ or $(*,*)$.  Hence if Assumption 7.3  holds
\[ E(GP)_3\cong H(H^{*,*'}(X;\bZ/2)/(\tau),Sq^2).\]
\begin{lemma}
(Lemma 6.12 in [Ya2]) 
In $H^{*,*'}(X;\bZ/2)$, we have
\[ Sq^2(a')=\rho Q_0(a'),\quad Sq^2(Q_0(a'))=Q_1(a').\]
\end{lemma}
\begin{proof}
By the modified Cartan formula
([Vo])
\[ 0=Sq^2(\rho^3)=Sq^2(\tau a')=\tau Sq^2(a')+Sq^2(\tau)a'+
 \tau Sq^1(\tau)Sq^1(a').\]
Hence $\tau Sq^2(a')=\tau Sq^1(\tau)Sq^1(a')=\tau \rho Q_0(a')$.  Since
$H^{*,*'}(X;\bZ/2)$ is $\tau$-torsion free, we have the first equation.  The second equation follows from
\[ Q_1(a')=Q_0Sq^2(a')+Sq^2Q_0(a')\]
and $Q_0Sq^2(a')=
\rho Q_0Q_0(a')=0$.
\end{proof}
Thus we have Lemma 7.5.
\begin{lemma} If Assumption 7.3 is satisfied, then we have the isomorphism  
\[E(GP)_3\cong \bZ/2[\rho]/(\rho^3)\otimes H(S(t)/(b);Sq^2).\]
\end{lemma}
\begin{proof}
The following submodule in $H^*(X;H_{\bZ/2}^{*'})$
\[A=H^{*,*'}(R_2;\bZ/2)/(\tau)
\cong \bZ/2\{1, \rho,\rho^2,a',Q_0(a'),\rho Q_0(a'),Q_1(a')\}.\]
is closed under $Sq^2$. Its $Sq^2$-homology is 
$H(A;Sq^2)\cong \bZ/2\{1,\rho,\rho^3\}$.  Note 
$S(t)/(b)$ is also closed under $Sq^2$.since $H^*(G/T;\bZ/2)$ is closed under $Sq^2$.  By the Kunneth formula, we have the lemma.
\end{proof}
\begin{lemma}  
If Assumption 7.3 is satisfied, then we have
\[ E(GP)_3\cong E(GP)_{\infty}\cong E(BW)_{\infty}.\]
\end{lemma}
\begin{proof}
We consider the map which induced from the product
\[ \mu: G_k\times \bG_k/B_k = G_k\times X\to X.\]
 It is known [Ya5] that  $W^*(G_k)\cong W^*(G_k/T_k)$.
At first we recall 
\[ H^{*,*'}(G_k/B_k\times X;\bZ/2)\cong H^{*,*'}(G_k/B_k;\bZ/2)
\otimes H^{*,*'}(X;\bZ/2) \]
since $H^{*,*'}(G_k/B_k;\bZ/2)\cong H^{*,*'}(\bar X;\bZ/2)$
and $\bar X$ is cellular.
By using Kunneth formula, inductively we get
\[E(GP)_r(G_k/B_k\times X)\cong E(GP)_r(G_k/B_k)\otimes
E(GP)_r(X)\]
where $E(GP)_r(Y)$ is the spectral sequence for a space $Y$.
Similar fact  holds for $E(BW)_r$.  
Thus we get
 $ W^*(G_k\times X)\cong W^*(G_k)\otimes W^*(X).$

Define  that $x\in W^*(X)$ is primitive if
\[ \mu^*(x)=res^*(x)\otimes 1+1\otimes x \quad res : W^*(X)\to W^*(\bar X)\]
Of course, if $x$ is primitive, then $d_r(x)$ is primitive.
We can take generators $z_i$ primitive (otherwise, add some
non-primitive elements).  If $d_3(z_i)\not =0$, then 
it is $\rho z$ for $z\in H^*(\bar X;H_{\bZ/2}^*)$.  But
\[\mu^*(\rho z)=1\otimes \rho z+\rho\otimes z+... \]
is not primitive.  Similarly, we can prove the lemma.
\end{proof}

\end{document}